\documentclass[11pt,a4paper,reqno]{amsart}
\usepackage{amsmath,amssymb,amsfonts,amsthm,mathtools}
\usepackage{graphicx}
\usepackage[hidelinks]{hyperref}
\usepackage{enumitem}
\usepackage{booktabs,longtable}
\usepackage{xcolor}
\usepackage[margin=1.05in]{geometry}
\usepackage{microtype}

\newtheorem{theorem}{Theorem}[section]
\newtheorem{proposition}[theorem]{Proposition}
\newtheorem{lemma}[theorem]{Lemma}
\newtheorem{corollary}[theorem]{Corollary}

\newtheorem{problem}[theorem]{Problem}
\theoremstyle{remark}
\newtheorem{remark}[theorem]{Remark}

\newcommand{\C}{\mathbb C}
\newcommand{\D}{\mathbb D}
\newcommand{\T}{\mathbb T}
\newcommand{\Spec}{\operatorname{Spec}}
\newcommand{\disc}{\operatorname{disc}}
\newcommand{\Th}{\Theta}

\newcommand{\qpoch}[2]{(#1;#2)_\infty}
\newcommand{\tr}{\mathrm{tr}}

\newcommand{\maybeincludegraphics}[2][]{%
\IfFileExists{#2}{\includegraphics[#1]{#2}}{%
\fbox{\begin{minipage}[c][0.23\textheight][c]{0.38\textwidth}\centering\small Missing external figure:\\ \texttt{#2}\end{minipage}}}}

\title[Complex spectrum of the partial theta function]{Complex spectrum of the partial theta function}
\author{Boris Shapiro}
\address{Department of Mathematics,
   Stockholm University, Stockholm, S-10691, Sweden}
\email{shapiro@math.su.se}
\date{\today}

\begin{document}
\begin{abstract}
We study the complex spectrum of the partial theta function
\[
        \Theta(q,x)=\sum_{j=0}^{\infty}q^{j(j+1)/2}x^j,
        \qquad |q|<1,
\]
where a spectral value is a parameter for which \(\Theta(q,\cdot)\) has a multiple zero.  Since the function is defined here only for \(|q|<1\), all spectral values are strictly inside the unit disk; boundary points on \(|q|=1\) occur only as accumulation points of the spectrum.  The paper combines two complementary points of view.  Near the unit circle we prove that every point of \(|q|=1\) is an accumulation point of the spectrum; the proof uses explicit spectral factors of truncations, the Jacobi triple product, and a boundary-window lifting argument near roots of unity.  Inside a fixed subdisk, illustrated for \(|q|\leq 0.8\), the true spectrum is locally finite and must be separated carefully from the much larger branch loci of truncations and Jensen polynomials.  We give a truncation-seeded Newton procedure which produces a discrete list of candidate spectral values, explain the caustic/escaping-root mechanism in finite approximants, and record numerical monodromy experiments using a radial convention: for a spectral point \(q_*\), roots are labelled at the point \(0.1q_*/|q_*|\) on the small circle and then continued along the straight radial segment to \(q_*\).  This convention gives a coherent set of collision labels in the disk, treats negative real spectral values from the base point \(-0.1\), and leads to a preliminary rational-direction heuristic for radial monodromy.
\end{abstract}
\maketitle

\section{Introduction}
Let
\[
        \Theta(q,x):=\sum_{j=0}^{\infty}q^{j(j+1)/2}x^j,
        \qquad |q|<1,
\]
and define its complex spectrum by
\[
        \Spec(\Theta):=\{q\in\D:\ \exists x\in\C\text{ such that }
        \Theta(q,x)=\partial_x\Theta(q,x)=0\}.
\]
Thus, by definition, \(\Spec(\Theta)\subset\D\).  Whenever we speak about spectral behaviour on the unit circle, this means behaviour of the closure of the spectrum in the closed disk, not the existence of spectral parameters with \(|q|=1\).
Equivalently, \(\Spec(\Theta)\) is the branch locus of the analytic curve \(\Theta(q,x)=0\) under projection to the \(q\)-plane.  The partial theta function is the one-sided analogue of the Jacobi theta function: the quadratic exponent is retained, but the summation is restricted to non-negative indices.  This asymmetry destroys the modular symmetry of the full theta series and produces a rich zero geometry.  Partial theta functions occur in \(q\)-series, questions on roots of formal power series, asymptotic problems in combinatorics, and real-rootedness questions of Hardy--Petrovitch--Hutchinson type; see Sokal's lectures \cite{SokalNotes}, Edrei--Saff--Varga \cite{EdreiSaffVarga}, and Kostov--Shapiro \cite{KostovShapiro}.  Kostov's subsequent work on the positive and complex spectra and on double zeros shows that the spectral set has a global asymptotic structure \cite{KostovAsymptotics,KostovSpectrum,KostovDoubleZeros}.

There are two apparently different numerical phenomena.  First, spectra of truncations and Jensen polynomials show many points near the unit circle and suggest boundary accumulation.  Second, in a compact subdisk such as \(|q|\leq0.8\), the true spectrum is discrete; raw finite discriminant plots show arcs and caustics which should not be confused with arcs of genuine spectral values.  These observations are compatible rather than contradictory.  The spectrum is locally finite in \(\D\), but may have every boundary point of \(\D\) as an accumulation point.  Thus one must distinguish the boundary asymptotic problem from the interior approximation problem.

For the degree \(n\) truncation
\[
        \Theta_n(q,x):=\sum_{j=0}^{n}q^{j(j+1)/2}x^j
\]
we write
\[
        \Spec_n:=\{q\in\C^*:\ \disc_x\Theta_n(q,x)=0\}.
\]
Jensen polynomials provide another finite-dimensional approximation; they are useful, but their branch loci contain many points produced by double roots escaping to infinity in the \(x\)-plane.  The same caution applies to ordinary truncations.

\begin{figure}[h]
\centering
\maybeincludegraphics[scale=0.35]{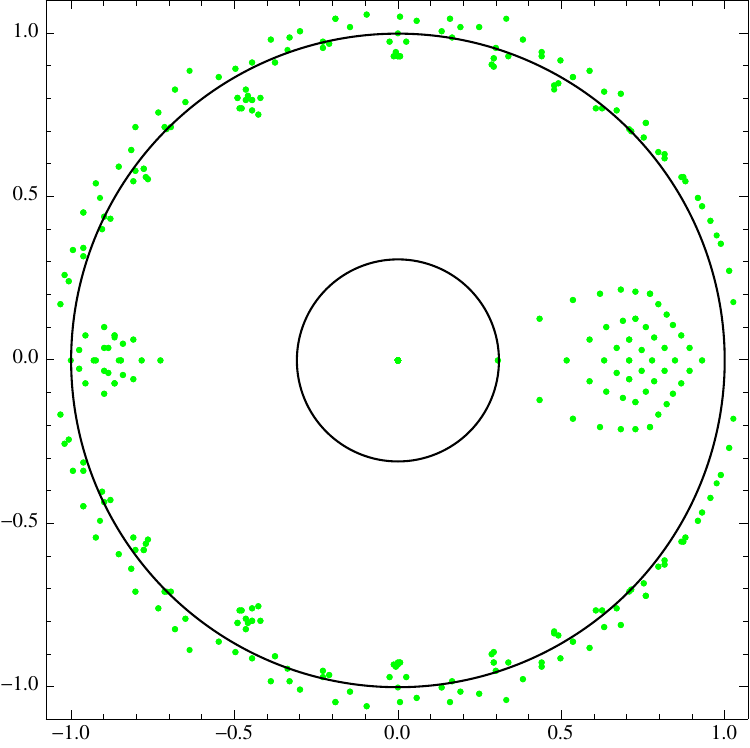}\hspace{1cm}
\maybeincludegraphics[scale=0.20]{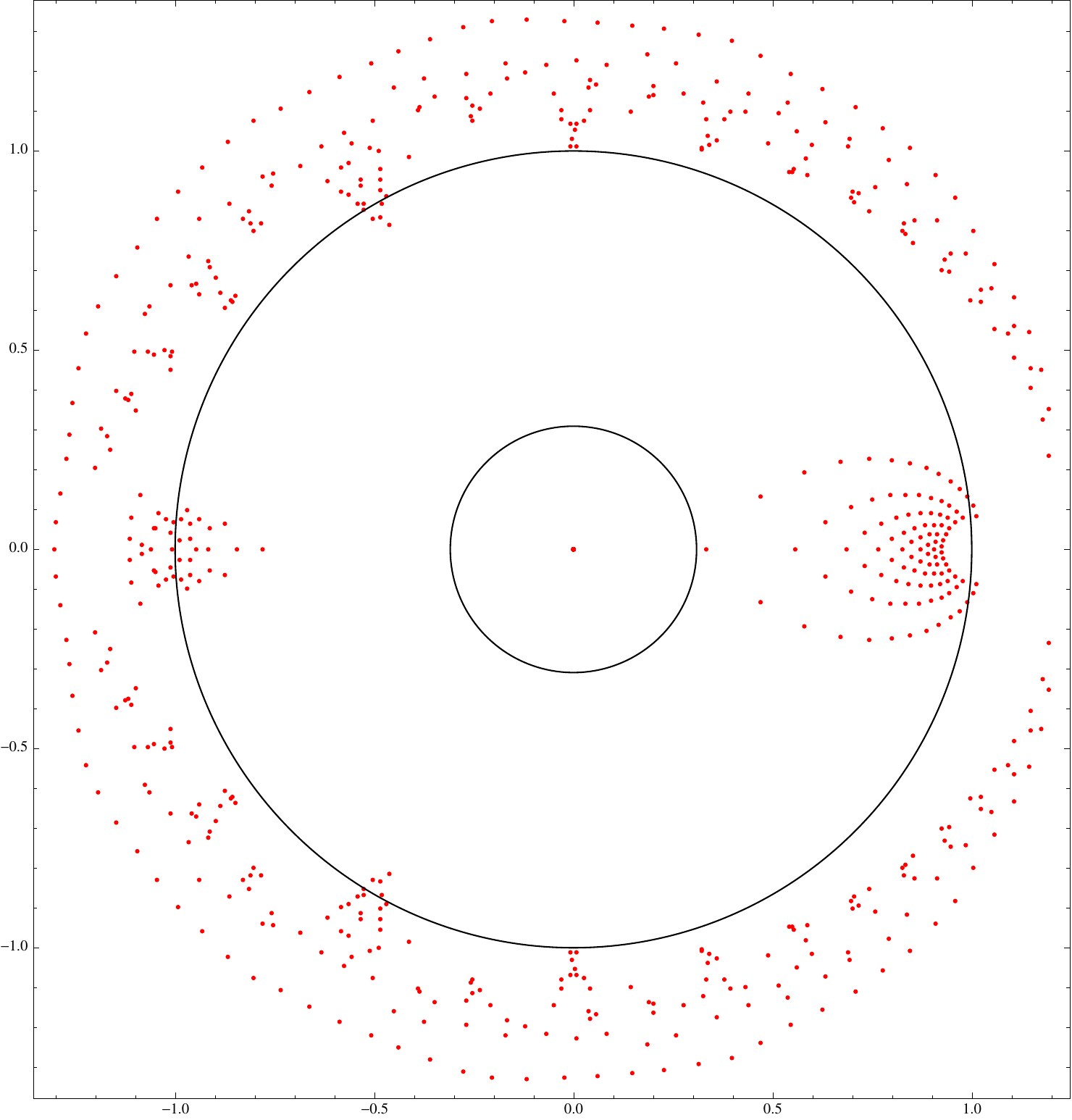}
\caption{Branch points for a degree-15 truncation of \(\Theta\) (left) and for the corresponding Jensen polynomial (right).  These pictures show finite approximation geometry; points in such plots are genuine spectral approximants only when the associated double root in the \(x\)-plane remains controlled.}
\label{fig:degree15-spectra}
\end{figure}

This figure should be read only as a picture of finite branch loci, not as a reliable approximation to the actual spectrum in the open unit disk.  In particular, the arc-like clusters visible in the finite truncation and Jensen pictures need not survive as arcs of genuine spectral values.  The later interior computations in Sections~10--12 use a stricter test: a finite branch point is treated as a genuine spectral candidate only after Newton refinement of the infinite system and after checking that the associated double root in the \(x\)-plane remains bounded.  Figure~\ref{fig:degree15-spectra} is therefore motivational for the boundary and caustic phenomena, while Figures~\ref{fig:candidates-r08} and~\ref{fig:trunc-vs-candidates} give the more trustworthy approximation to the actual spectrum in \(|q|\leq0.8\).

\paragraph{Structure of the paper.}
Sections~2--7 contain the asymptotic boundary result.  The key finite-dimensional input is an explicit spectral factor for odd truncations whose zeros equidistribute on the unit circle.  The passage to the infinite function uses a moving-window formulation and a root-of-unity lifting argument.  Sections~8--17 discuss the interior problem: bounded-root criteria, even-degree truncation factorization, a numerical list in \(|q|\leq0.8\), monodromy of the first and second spectral layers, negative-real and rational-ray monodromy heuristics, possible rational-angle collections, and the compatibility of the boundary and interior pictures.

\section{Basic convergence and finite approximants}
This section records the parts of Forsg\aa rd's notes which are used later, see \cite{ForsgardNotes}.  We give complete proofs, with a slightly more elementary Rouch\'e-type proof of the outer localization statement.

\subsection{Uniform convergence of truncations}
Set
\[
        p^{\Theta}_n(q,z)=\sum_{j=0}^n q^{j(j+1)/2}z^j,
        \qquad
        p^F_n(q,z)=\sum_{j=0}^n\frac{q^{j(j+1)/2}}{j!}z^j,
\]
and
\[
        F(q,z)=\sum_{j=0}^{\infty}\frac{q^{j(j+1)/2}}{j!}z^j.
\]

\begin{lemma}[Uniform convergence on polydiscs]\label{lem:uniform-polydisc}
For every \(0<\delta<1\) and every \(\rho>0\), the truncations \(p^\Theta_n\) and \(p^F_n\) converge respectively to \(\Theta\) and \(F\) uniformly on
\[
        D(\delta,\rho)=\{(q,z): |q|\le\delta,\ |z|\le\rho\}.
\]
The same is true after differentiating with respect to \(z\).
\end{lemma}

\begin{proof}
For \((q,z)\in D(\delta,\rho)\),
\[
        |q|^{j(j+1)/2}|z|^j\le \delta^{j(j+1)/2}\rho^j,
\]
and the numerical series \(\sum_{j\ge0}\delta^{j(j+1)/2}\rho^j\) converges.  The Weierstrass test gives uniform convergence of \(p^\Theta_n\).  The proof for \(p^F_n\) is even easier because of the extra factor \(1/j!\).  Differentiation in \(z\) only inserts a factor \(j\), and
\[
        \sum_{j\ge1}j\delta^{j(j+1)/2}\rho^{j-1}<\infty,
\]
so the derivative series also converge uniformly.
\end{proof}

\begin{lemma}[Compact lifting of bounded multiple zeros]\label{lem:compact-lifting}
Let \(f_n\) be either \(p^\Theta_n\) or \(p^F_n\), and let \(f\) be the corresponding limit \(\Theta\) or \(F\).  Suppose that \((q_m,z_m)\to(q_*,z_*)\), with \(|q_*|<1\), and that \(z_m\) is a multiple zero of \(z\mapsto f_{n(m)}(q_m,z)\), where \(n(m)\to\infty\).  Then \(z_*\) is a multiple zero of \(z\mapsto f(q_*,z)\).
\end{lemma}

\begin{proof}
Choose \(\delta<1\) and \(\rho>0\) such that \(|q_m|\le\delta\) and \(|z_m|\le\rho\) for all large \(m\).  Lemma~\ref{lem:uniform-polydisc} gives
\[
        \sup_{D(\delta,\rho)}|f_n-f|\to0,
        \qquad
        \sup_{D(\delta,\rho)}|\partial_z f_n-\partial_z f|\to0.
\]
Since \(f_{n(m)}(q_m,z_m)=\partial_zf_{n(m)}(q_m,z_m)=0\), we obtain
\[
        f(q_m,z_m)\to0,
        \qquad
        \partial_zf(q_m,z_m)\to0.
\]
Passing to the limit gives \(f(q_*,z_*)=\partial_zf(q_*,z_*)=0\).
\end{proof}

\begin{remark}
Lemma~\ref{lem:compact-lifting} is useful for interior compact limits.  It is not enough for the unit-circle problem, because the relevant double roots of high truncations occur at \(x\)-values whose natural scale moves with the degree.
\end{remark}

\subsection{The Jensen-polynomial identity}
For an entire function \(f(z)=\sum_{j\ge0}a_jz^j/j!\), its \(n\)-th Jensen polynomial is
\[
        g^f_n(z)=\sum_{j=0}^n\binom nj a_j\frac{z^j}{n^j}.
\]
In the present situation this gives
\[
        g^\Theta_n(q,z)=n!\sum_{j=0}^n\frac{q^{j(j+1)/2}}{(n-j)!n^j}z^j,
        \qquad
        g^F_n(q,z)=\sum_{j=0}^n\binom nj q^{j(j+1)/2}\frac{z^j}{n^j}.
\]

\begin{lemma}[Jensen--truncation identity]\label{lem:jensen-identity}
For every \(n\ge0\),
\begin{equation}\label{eq:jensen-pf-identity}
        n^n z^n g^\Theta_n(q,z^{-1})
        =n!q^{n(n+1)/2}p^F_n(q,nq^{-n-1}z).
\end{equation}
Consequently, for fixed \(q\ne0\), the polynomial \(g^\Theta_n(q,\cdot)\) has a multiple zero if and only if \(p^F_n(q,\cdot)\) has a multiple zero.
\end{lemma}

\begin{proof}
Starting from the left-hand side before multiplying by \(n^n\), one has
\begin{align*}
 z^n g^\Theta_n(q,z^{-1})
 &=n!\sum_{j=0}^n\frac{q^{j(j+1)/2}}{(n-j)!n^j}z^{n-j}  \\
 &=n!\sum_{j=0}^n\frac{q^{(n-j)(n-j+1)/2}}{j!n^{n-j}}z^j  \\
 &=\frac{n!q^{n(n+1)/2}}{n^n}
     \sum_{j=0}^n\frac{q^{j(j+1)/2}}{j!}(nq^{-n-1}z)^j.
\end{align*}
Multiplying by \(n^n\) gives \eqref{eq:jensen-pf-identity}.  The map \(z\mapsto nq^{-n-1}z\), followed by inversion in the Jensen variable, is biholomorphic away from \(z=0\), and neither polynomial has a multiple zero created at the transformation point.  Hence multiple zeros correspond.
\end{proof}

\begin{remark}
This identity explains why the Jensen spectrum for \(\Theta\) and the truncation spectrum for \(F\) coincide in the \(q\)-plane.  It is useful for interpreting the figures, but it does not by itself solve the boundary accumulation problem for the actual spectrum of \(\Theta\).
\end{remark}

\subsection{The eta-cubed product}
Define
\[
        \psi(q)=\sum_{k=0}^{\infty}(-1)^k(2k+1)q^{k(k+1)/2}.
\]

\begin{lemma}[Jacobi product form]\label{lem:eta-cubed}
For \(|q|<1\),
\[
        \psi(q)=\prod_{\nu=1}^{\infty}(1-q^\nu)^3.
\]
\end{lemma}

\begin{proof}
Use Jacobi's triple product in the form
\[
 \sum_{j=-\infty}^{\infty}(-1)^j a^j q^{j(j-1)/2}
 =\prod_{\nu=1}^{\infty}(1-aq^{\nu-1})(1-a^{-1}q^\nu)(1-q^\nu).
\]
Both sides vanish at \(a=1\).  Dividing by \(1-a\) and letting \(a\to1\), the product side tends to
\[
        \prod_{\nu=1}^{\infty}(1-q^\nu)^3.
\]
On the series side the limit is
\[
        -\left.\frac{d}{da}\right|_{a=1}
        \sum_{j=-\infty}^{\infty}(-1)^j a^j q^{j(j-1)/2}
        =-\sum_{j=-\infty}^{\infty}(-1)^j j q^{j(j-1)/2}.
\]
Pairing the terms \(j=k+1\) and \(j=-k\), \(k\ge0\), gives exactly
\[
        \sum_{k=0}^{\infty}(-1)^k(2k+1)q^{k(k+1)/2}.
\]
This proves the identity.
\end{proof}

\subsection{Outer localization of truncation spectra}
The notes prove an outer localization theorem with explicit thresholds.  The following proof is a self-contained Rouch\'e version sufficient for the present paper.

\begin{lemma}[A Rouch\'e stability lemma]\label{lem:rouche-stability}
For every \(\eta>0\) there exists \(\varepsilon_0>0\) with the following property.  If
\[
        P_N(w)=1+w^N+R_N(w),
        \qquad
        R_N(w)=\sum_{j=1}^{N-1}a_jw^j,
\]
and \(\sum_{j=1}^{N-1}|a_j|<\varepsilon_0\), then all zeros of \(P_N\) are simple and each lies within distance \(\eta/N\) of a zero of \(1+w^N\).
\end{lemma}

\begin{proof}
Choose \(\eta>0\) so small that the disks of radius \(\eta/N\) around the roots of \(1+w^N\) are pairwise disjoint for all \(N\ge2\).  On the boundary of any such disk, Taylor expansion of \(w^N\) at the corresponding simple root gives
\[
        |1+w^N|\ge c_\eta>0,
\]
with \(c_\eta\) independent of \(N\).  On the same boundary, \(|w|\le 1+\text{const}/N\), hence
\[
        |R_N(w)|\le \left(1+\frac{\text{const}}N\right)^N\sum_{j=1}^{N-1}|a_j|
        \le C_\eta\sum_{j=1}^{N-1}|a_j|.
\]
If \(\varepsilon_0<c_\eta/C_\eta\), Rouch\'e's theorem gives exactly one zero of \(P_N\), counted with multiplicity, in each disk.  Since there are \(N\) disjoint disks and \(P_N\) has degree \(N\), these are all zeros.  Each disk contains one zero counted with multiplicity, so the zero is simple.
\end{proof}

\begin{theorem}[Outer localization of truncation spectra]\label{thm:outer-localization}
Let \(q_m\) be a sequence with \(|q_m|\ge1\), and suppose that \(n(m)\to\infty\) and \(q_m\in\Spec_{n(m)}\).  If \(|q_m|\) has a limit, then this limit is \(1\).  Equivalently, for every \(\varepsilon>0\) all sufficiently high truncations have no spectral values in \(|q|\ge1+\varepsilon\).
\end{theorem}

\begin{proof}
Let \(N=n(m)\).  Normalize
\[
        \Theta_N(q,z)=\sum_{j=0}^Nq^{j(j+1)/2}z^j
\]
by setting \(z=q^{-(N+1)/2}w\).  Up to a non-zero scalar factor, the normalized polynomial is
\[
        P_N(q,w)=1+w^N+
        \sum_{j=1}^{N-1}q^{j(j-N)/2}w^j.
\]
If \(|q|\ge1+\varepsilon\), then
\[
        \sum_{j=1}^{N-1}|q|^{-j(N-j)/2}
        \le (N-1)(1+\varepsilon)^{-(N-1)/2}\longrightarrow0.
\]
For all sufficiently large \(N\), Lemma~\ref{lem:rouche-stability} applies and shows that \(P_N(q,\cdot)\), and hence \(\Theta_N(q,\cdot)\), has only simple zeros.  Thus no spectral values occur in \(|q|\ge1+\varepsilon\) for high degree.  This proves the assertion.
\end{proof}

\begin{remark}
Forsg\aa rd's notes \cite{ForsgardNotes} give more explicit sufficient inequalities, separated according to parity.  The elementary proof above is less sharp but avoids the discriminant-amoeba argument and is enough for the asymptotic localization needed here.
\end{remark}

\section{An explicit spectral factor for odd truncations}
The following elementary observation is the main source of the unit circle in the finite spectral pictures.

\begin{theorem}[Central spectral factor for odd truncations]\label{thm:central-factor}
Let \(n=2m+1\), and set
\[
        \Psi_m(q):=\sum_{k=0}^{m}(-1)^k(2k+1)q^{k(k+1)/2}.
\]
If \(q\ne0\) and \(\Psi_m(q)=0\), then \(q\in\Spec_{2m+1}\).  More precisely,
\[
        x=-q^{-(m+1)}
\]
is a multiple zero of \(\Theta_{2m+1}(q,x)\).
\end{theorem}

\begin{proof}
Put
\[
        x=q^{-(m+1)}y,
\]
and write
\[
        Q_m(q,y):=\Theta_{2m+1}(q,q^{-(m+1)}y)
        =\sum_{j=0}^{2m+1}q^{j(j-2m-1)/2}y^j.
\]
The coefficient of \(y^j\) equals the coefficient of \(y^{2m+1-j}\), because
\[
        j(j-2m-1)=(2m+1-j)((2m+1-j)-2m-1).
\]
Thus \(Q_m\) is palindromic of odd degree.  Consequently \(Q_m(q,-1)=0\).

It remains to impose the vanishing of the derivative.  Pairing the terms with indices \(j\) and \(2m+1-j\), one obtains
\begin{align*}
        \partial_y Q_m(q,-1)
        &=\sum_{j=0}^{2m+1}j(-1)^{j-1}q^{-j(2m+1-j)/2}  \\
        &=\sum_{j=0}^{m}(-1)^j(2m+1-2j)q^{-j(2m+1-j)/2}.
\end{align*}
Now put \(k=m-j\).  Since
\[
        \frac{m(m+1)}2-\frac{j(2m+1-j)}2=\frac{k(k+1)}2,
\]
we get
\[
        \partial_y Q_m(q,-1)
        =(-1)^m q^{-m(m+1)/2}\Psi_m(q).
\]
Therefore \(Q_m(q,-1)=\partial_y Q_m(q,-1)=0\) exactly when \(\Psi_m(q)=0\).  Since \(\partial_x=q^{m+1}\partial_y\), the same condition says that \(x=-q^{-(m+1)}\) is a multiple zero of \(\Theta_{2m+1}(q,x)\).
\end{proof}

\begin{remark}
This theorem gives a concrete diagnostic for the figures: the zeros of \(\Psi_m\) must occur among the branch points plotted for \(\Theta_{2m+1}\).  Numerically they should form a nearly circular cloud around \(|q|=1\).
\end{remark}

\begin{figure}[h]
\centering
\maybeincludegraphics[width=0.55\textwidth]{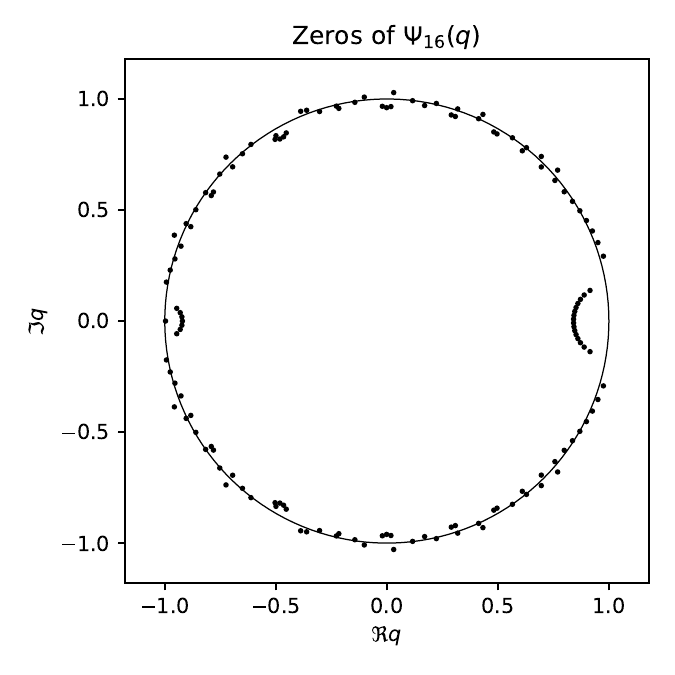}
\caption{Zeros of \(\Psi_{16}\), hence a distinguished spectral subfamily of the truncation \(\Theta_{33}\).  Theorem~\ref{thm:equidistribution} predicts convergence to Haar measure on the unit circle.}
\end{figure}

\section{Density on the unit circle for the truncation-spectral subfamily}
The factor \(\Psi_m\) is the \(m\)-th non-trivial Taylor section of the Jacobi product from Lemma~\ref{lem:eta-cubed}.  This explains why roots of unity, and hence the whole unit circle, appear in the limiting picture.

\begin{theorem}[Equidistribution of the central spectral factor]\label{thm:equidistribution}
Let \(N_m=m(m+1)/2\), and let
\[
        \nu_m:=\frac1{N_m}\sum_{\Psi_m(\alpha)=0}\delta_{\alpha},
\]
where zeros are counted with multiplicity.  Then
\[
        \nu_m\xrightarrow[m\to\infty]{w^*}\frac{d\theta}{2\pi}\quad\text{on }\T.
\]
Equivalently, for every fixed arc \(I\subset\T\) whose endpoints are irrelevant to the limiting measure,
\[
        \frac1{N_m}\#\{\alpha:\Psi_m(\alpha)=0,\ \arg\alpha\in I\}
        \longrightarrow \frac{|I|}{2\pi},
\]
and all but \(o(N_m)\) zeros satisfy \(1-\varepsilon<|\alpha|<1+\varepsilon\) for every fixed \(\varepsilon>0\).
\end{theorem}

\begin{proof}
The polynomial \(\Psi_m\) has degree \(N_m\), constant coefficient \(1\), and leading coefficient \((-1)^m(2m+1)\).  On \(|q|=1\),
\[
        \|\Psi_m\|_{\T}\le \sum_{k=0}^m(2k+1)=(m+1)^2.
\]
Hence
\[
        \log\frac{\|\Psi_m\|_{\T}}{\sqrt{|a_0a_{N_m}|}}
        =O(\log m)=o(N_m).
\]
The Erd\H os--Tur\'an angular discrepancy theorem \cite{ErdosTuran} therefore gives uniform angular distribution of the zeros.

It remains to note that the zeros concentrate radially near \(\T\).  Jensen's formula gives, for fixed \(0<r<1\),
\[
        \#\{\alpha:\Psi_m(\alpha)=0,\ |\alpha|\le r\}
        \le \frac{\log\|\Psi_m\|_{\T}}{\log(1/r)}=O_r(\log m)=o(N_m).
\]
For the exterior, apply the same estimate to the reciprocal polynomial
\[
        q^{N_m}\Psi_m(1/q)/((-1)^m(2m+1)),
\]
whose constant coefficient is \(1\) and whose unit-circle norm is \(O(m)\).  This gives \(o(N_m)\) zeros in \(|q|\ge R\) for every fixed \(R>1\).  Radial concentration together with angular equidistribution proves the weak convergence to Haar measure on \(\T\).
\end{proof}

\begin{corollary}[Every point of the unit circle is seen by truncation spectra]
For every \(\zeta\in\T\), there exist integers \(m_s\to\infty\) and points \(q_s\in\Spec_{2m_s+1}\) such that
\[
        q_s\to\zeta.
\]
Indeed, one may choose \(q_s\) among the zeros of \(\Psi_{m_s}\).
\end{corollary}

\begin{remark}[Asymptotic angular density]
For this explicit truncation-spectral subfamily the limiting density is
\[
        \rho(\theta)=\frac1{2\pi}.
\]
This theorem concerns a distinguished factor of the truncation discriminant.  It does not claim equidistribution for the full discriminant of \(\Theta_n\), whose degree is much larger.
\end{remark}

\section{Exact moving-window reformulation for the infinite function}
We now push the problem closer to the actual spectrum of \(\Theta\).  The following change of variables is the most useful way to see the obstruction.

For \(N\ge1\), set
\[
        x=-q^{-N}e^u.
\]
Then
\begin{align*}
\Theta(q,-q^{-N}e^u)
&=\sum_{j=0}^{\infty}(-1)^j q^{j(j+1)/2-Nj}e^{uj}.
\end{align*}
Writing \(j=N+s\), we obtain the exact identity
\begin{equation}\label{eq:moving-window}
        \Theta(q,-q^{-N}e^u)
        =(-1)^Nq^{-N(N-1)/2}e^{Nu}H_N(q,u),
\end{equation}
where
\begin{equation}\label{eq:HN-def}
        H_N(q,u)=\sum_{s=-N}^{\infty}(-1)^s q^{s(s+1)/2}e^{su}.
\end{equation}

\begin{proposition}[Double roots in the moving window]\label{prop:moving-window-double}
For \(q\ne0\) and \(x=-q^{-N}e^u\), the two equations
\[
        \Theta(q,x)=0,
        \qquad
        \partial_x\Theta(q,x)=0
\]
are equivalent to
\[
        H_N(q,u)=0,
        \qquad
        \partial_uH_N(q,u)=0.
\]
\end{proposition}

\begin{proof}
The factor multiplying \(H_N\) in \eqref{eq:moving-window} is non-zero.  Thus \(\Theta(q,x)=0\) is equivalent to \(H_N(q,u)=0\).  Since \(x\partial_x=\partial_u\), the equation \(\partial_x\Theta(q,x)=0\) is equivalent to \(\partial_u\Theta(q,-q^{-N}e^u)=0\).  Differentiating \eqref{eq:moving-window} with respect to \(u\) gives a factor times \(NH_N+\partial_uH_N\).  Under the first equation \(H_N=0\), the derivative condition is exactly \(\partial_uH_N=0\).
\end{proof}

For fixed \(|q|<1\), the functions \(H_N\) converge uniformly on compact subsets in \(u\) to the bilateral theta series
\[
        H(q,u)=\sum_{s=-\infty}^{\infty}(-1)^s q^{s(s+1)/2}e^{su}.
\]
By Jacobi's triple product,
\begin{equation}\label{eq:bilateral-product}
        H(q,u)=\prod_{\nu=1}^{\infty}(1-q^\nu)(1-e^u q^\nu)(1-e^{-u}q^{\nu-1}).
\end{equation}
In particular \(H(q,0)=0\) and
\[
        \partial_uH(q,0)=\prod_{\nu=1}^{\infty}(1-q^\nu)^3=\psi(q).
\]
Thus the factor \(\Psi_m\) found above is a finite-section version of the derivative of the bilateral limiting theta function at \(u=0\).

\begin{remark}[Why compact convergence is not enough]
For fixed \(|q|<1\), the negative tail
\[
        H(q,u)-H_N(q,u)=\sum_{s=-\infty}^{-N-1}(-1)^s q^{s(s+1)/2}e^{su}
\]
is exponentially small in \(N^2\) on compact \(u\)-sets.  However, the desired boundary accumulation requires \(|q|\to1\) as \(N\to\infty\).  In that regime the estimate degenerates, and this is precisely why the truncation theorem above does not immediately lift to the actual spectrum.
\end{remark}

\section{Boundary-window formulation of the lifting problem}
The preceding section suggests that the full problem should be attacked through the critical values of \(H_N\).  Near \(u=0\) write
\[
        A_N(q)=H_N(q,0),\qquad
        B_N(q)=\partial_uH_N(q,0),\qquad
        C_N(q)=\partial_u^2H_N(q,0).
\]
The truncation calculation corresponds to replacing \(H_N\) by its symmetric finite part
\[
        H_N^{\mathrm{sym}}(q,u)=\sum_{s=-N}^{N-1}(-1)^sq^{s(s+1)/2}e^{su}.
\]
For this symmetric finite part one has
\[
        H_N^{\mathrm{sym}}(q,0)=0,
        \qquad
        \partial_uH_N^{\mathrm{sym}}(q,0)=(-1)^{N-1}q^{-N(N-1)/2}\Psi_{N-1}(q),
\]
which is Theorem~\ref{thm:central-factor} in moving-window coordinates.

For the infinite function, one should instead solve
\[
        H_N(q,u)=\partial_uH_N(q,u)=0.
\]
A natural one-variable reduction is the following.

\begin{proposition}[Critical-value reduction]\label{prop:critical-value-reduction}
Let \(\Omega_N\) be a region in the \(q\)-plane and suppose that, for \(q\in\Omega_N\), \(\partial_u^2H_N(q,0)\) is bounded away from zero while \(\partial_uH_N(q,0)\) is sufficiently small.  Then there is a unique small analytic function \(u_N(q)\) such that
\[
        \partial_uH_N(q,u_N(q))=0.
\]
In this region, \(q\in\Spec(\Theta)\) with a double zero in the window \(x=-q^{-N}e^u\) if and only if
\[
        \Phi_N(q):=H_N(q,u_N(q))=0.
\]
Moreover,
\[
        \Phi_N(q)=A_N(q)-\frac{B_N(q)^2}{2C_N(q)}+\text{higher order terms},
\]
where the higher order terms are controlled by higher \(u\)-derivatives of \(H_N\).
\end{proposition}

\begin{proof}
The existence and analyticity of \(u_N(q)\) follow from the implicit function theorem applied to \(\partial_uH_N(q,u)=0\) at \(u=0\).  Proposition~\ref{prop:moving-window-double} then gives the equivalence with \(\Phi_N(q)=0\).  The displayed expansion is the usual Taylor expansion at the critical point: to first order, \(u_N(q)=-B_N(q)/C_N(q)\), and substitution into
\[
        H_N(q,u)=A_N(q)+B_N(q)u+\frac12C_N(q)u^2+\cdots
\]
gives the stated expression.
\end{proof}

The lifting statement needed below is therefore not a conjectural assertion about every boundary point separately.  It is enough to prove it at roots of unity, because roots of unity are dense on the unit circle.  We record this elementary reduction before proving the required root-of-unity lifting in the next section.

\begin{proposition}[Boundary-window reduction]\label{prop:boundary-window-reduction}
Suppose that for every root of unity \(\omega\in\T\) there exist integers \(N_s\to\infty\), parameters \(Q_s\in\D\), and small numbers \(u_s\in\C\), such that
\[
        Q_s\to\omega,
        \qquad
        H_{N_s}(Q_s,u_s)=\partial_uH_{N_s}(Q_s,u_s)=0.
\]
Then every point of \(\T\) is an accumulation point of \(\Spec(\Theta)\).
\end{proposition}

\begin{proof}
By Proposition~\ref{prop:moving-window-double}, each solution of the displayed moving-window system gives a genuine multiple zero
\[
        X_s=-Q_s^{-N_s}e^{u_s}
\]
of \(\Theta(Q_s,\cdot)\).  Hence each \(Q_s\) belongs to \(\Spec(\Theta)\).  Now let \(\zeta\in\T\).  Choose roots of unity \(\omega_r\to\zeta\).  Applying the hypothesis to \(\omega_r\), choose a spectral point \(Q_r\) with \(|Q_r-\omega_r|<1/r\).  Then \(Q_r\to\zeta\), proving the claim.
\end{proof}

\begin{remark}[Role of the moving-window formulation]
The initial truncation argument only produces double zeros of \(\Theta_{2N-1}\).  The moving-window identity \eqref{eq:moving-window} converts the genuine infinite problem into a precise pair of equations for \(H_N\).  The next section controls these equations in the rational boundary regime \(q\to\omega\), where the omitted one-sided tail and the bilateral theta factor have comparable exponential size.
\end{remark}

\begin{remark}[The plan]
The next section proves the root-of-unity lifting required in Proposition~\ref{prop:boundary-window-reduction}.  The proof uses the exact moving-window identity, the product formula for the bilateral theta function, and a two-variable Rouch\'e argument in a boundary layer of width $(\log N)/N$.  Thus the reduction above becomes an unconditional proof of unit-circle accumulation.
\end{remark}

\section{Root-of-unity boundary scaling and lifting}
In this section we finish the accumulation argument at roots of unity.  The result is local at a fixed root of unity, but density of roots of unity on $\T$ then gives the full unit-circle accumulation theorem.

Let $\omega$ be a primitive $b$-th root of unity.  We use the boundary scaling
\begin{equation}\label{eq:root-unity-scaling-final}
        q=q_N(\tau):=\omega\exp(-\tau/N),
        \qquad
        u=v/N,
        \qquad
        \Re\tau>0.
\end{equation}
The distinguished value is
\begin{equation}\label{eq:tau0-def}
        \tau_0:=\frac{\pi}{b}.
\end{equation}
It is the point where the eta-cubed factor in the bilateral theta function and the first omitted one-sided tail in $H_N$ have the same exponential order.

\subsection{The one-sided tail}
Recall that
\[
        H_N(q,u)=H(q,u)+R_N(q,u),
        \qquad
        H(q,u)=\sum_{s=-\infty}^{\infty}(-1)^sq^{s(s+1)/2}e^{su},
\]
where
\begin{equation}\label{eq:tail-final-def}
        R_N(q,u)=\sum_{t=N}^{\infty}(-1)^tq^{t(t+1)/2}e^{-(t+1)u}.
\end{equation}
Fix a residue class $r$ modulo $b$, and put
\begin{equation}\label{eq:Er-def}
        E_r(\tau):=
        \sum_{\ell=0}^{\infty}(-1)^\ell
        \omega^{r\ell+\ell(\ell+1)/2}e^{-\tau\ell}.
\end{equation}
This is an analytic function in $\Re\tau>0$.

\begin{lemma}[Tail asymptotics]\label{lem:tail-asymptotics-final}
Let $N\equiv r\pmod b$.  Uniformly for $\tau$ in compact subsets of $\{\Re\tau>0\}$ and for $v$ in compact subsets of $\C$,
\begin{align}
        R_N(q_N(\tau),v/N)
        &=(-1)^Nq_N(\tau)^{N(N+1)/2}e^{-v}
          \big(E_r(\tau)+O(N^{-1})\big),\label{eq:tail-asym-final}\\
        \partial_uR_N(q_N(\tau),v/N)
        &=-N(-1)^Nq_N(\tau)^{N(N+1)/2}e^{-v}
          \big(E_r(\tau)+O(N^{-1})\big).\label{eq:tail-der-asym-final}
\end{align}
The same estimates hold after one differentiation with respect to $\tau$ and $v$.
\end{lemma}

\begin{proof}
Put $t=N+\ell$ in \eqref{eq:tail-final-def}.  The $\ell$-th term is the first term times
\[
        (-1)^\ell q_N(\tau)^{\ell N+\ell(\ell+1)/2}e^{-\ell v/N}
        =(-1)^\ell \omega^{r\ell+\ell(\ell+1)/2}e^{-\tau\ell}
          \left(1+O\left(\frac{\ell^2+\ell |v|}{N}\right)\right).
\]
The summands are dominated by a geometric sequence on compact subsets of $\Re\tau>0$.  Termwise summation gives \eqref{eq:tail-asym-final}.  Differentiation with respect to $u$ multiplies the $t$-th summand by $-(t+1)=-N(1+O((\ell+1)/N))$, giving \eqref{eq:tail-der-asym-final}.  The estimates for $\tau$- and $v$-derivatives follow in the same way, since the additional polynomial factors in $\ell$ are still geometrically summable.
\end{proof}

We shall need at least one non-zero tail multiplier.

\begin{lemma}[Choice of residue class]\label{lem:tail-multiplier-nonzero}
For every primitive $b$-th root of unity $\omega$, there is a residue class $r$ modulo $b$ such that
\[
        E_r(\tau_0)\ne0,
        \qquad \tau_0=\pi/b.
\]
\end{lemma}

\begin{proof}
Write $E_r(\tau_0)=\sum_{a=0}^{b-1}\omega^{ra}S_a$, where
\[
        S_a:=\sum_{p=0}^{\infty}(-1)^{a+pb}
        \omega^{(a+pb)(a+pb+1)/2}e^{-\tau_0(a+pb)}.
\]
If all $E_r(\tau_0)$ were zero, discrete Fourier inversion would give $S_a=0$ for every $a$.  But
\[
        S_0=1+\sum_{p=1}^{\infty}(-1)^{pb}
        \omega^{pb(pb+1)/2}e^{-\pi p},
\]
and therefore
\[
        |S_0-1|\le \sum_{p=1}^{\infty}e^{-\pi p}<1.
\]
Thus $S_0\ne0$, a contradiction.
\end{proof}

\subsection{Cusp asymptotics for the bilateral part}
The bilateral theta factor has the product representation
\begin{equation}\label{eq:H-product-final}
        H(q,u)=\prod_{\nu=1}^{\infty}(1-q^\nu)(1-e^uq^\nu)(1-e^{-u}q^{\nu-1}),
\end{equation}
while
\[
        \psi(q)=\partial_uH(q,0)=\prod_{\nu=1}^{\infty}(1-q^\nu)^3.
\]
The next lemma is the standard local form of this product at a rational cusp.  We include the proof because it is the estimate which turns the previous conditional statement into a proof.

\begin{lemma}[Uniform sine and eta asymptotics]\label{lem:cusp-asymptotics-final}
Fix $C>0$ and a compact set $V\Subset\C$.  Uniformly for
\[
        |\tau-\tau_0|\le C\frac{\log N}{N},
        \qquad v\in V,
\]
one has
\begin{align}
        H(q_N(\tau),v/N)
        &=\frac{\psi(q_N(\tau))}{N}
          \left(S_b(\tau,v)+O\left(\frac{\log N}{N}\right)\right),\label{eq:H-sine-final}\\
        \partial_uH(q_N(\tau),v/N)
        &=\psi(q_N(\tau))
          \left(\partial_vS_b(\tau,v)+O\left(\frac{\log N}{N}\right)\right),\label{eq:Hu-sine-final}
\end{align}
where
\begin{equation}\label{eq:Sb-final-def}
        S_b(\tau,v):=\frac{b\tau}{\pi}\sin\frac{\pi v}{b\tau}.
\end{equation}
Moreover, after fixing a residue class of $N$ modulo $2b$,
\begin{equation}\label{eq:Lambda-final-asym}
        \Lambda_N(\tau):=\frac{N(-1)^Nq_N(\tau)^{N(N+1)/2}}{\psi(q_N(\tau))}
        =A_N(\tau)\exp\{Nf_b(\tau)\},
\end{equation}
where
\begin{equation}\label{eq:fb-def}
        f_b(\tau)=\frac{\pi^2}{2b^2\tau}-\frac{\tau}{2},
\end{equation}
$A_N$ is holomorphic and non-vanishing in the same window, and
\begin{equation}\label{eq:AN-bound}
        \log |A_N(\tau)|=O(\log N),
        \qquad
        \partial_\tau\log A_N(\tau)=O(\log N)
\end{equation}
there, after a holomorphic choice of the logarithm.  The estimates also hold after one $\tau$- and $v$-derivative in \eqref{eq:H-sine-final}--\eqref{eq:Hu-sine-final}.
\end{lemma}

\begin{proof}
We use the standard notation
\[
        \qpoch{a}{Q}:=\prod_{k=0}^{\infty}(1-aQ^k).
\]
From the product formula \eqref{eq:H-product-final} and Lemma~\ref{lem:eta-cubed},
\begin{equation}\label{eq:qpoch-quotient}
        \frac{H(q,u)}{\psi(q)}=
        \frac{\qpoch{e^u q}{q}\qpoch{e^{-u}}{q}}{\qpoch{q}{q}^2}.
\end{equation}
Put $Q=q^b=e^{-b\tau/N}$ and $h=b\tau/N$.  Grouping the products in \eqref{eq:qpoch-quotient} into residue classes modulo $b$ gives a singular factor
\begin{equation}\label{eq:singular-factor}
        G_N(\tau,v):=
        \frac{\qpoch{e^{v/N}Q}{Q}\qpoch{e^{-v/N}}{Q}}{\qpoch{Q}{Q}^2}
\end{equation}
and a product $U_N(\tau,v)$ over the non-zero residue classes.  For $a=1,\ldots,b-1$, the corresponding factor of $U_N$ is
\[
        \frac{\qpoch{e^{v/N}q^a}{Q}\qpoch{e^{-v/N}q^a}{Q}}{\qpoch{q^a}{Q}^2}.
\]
Since $q^a$ stays uniformly away from $1$ in the present window, the logarithm of this factor is obtained by a Taylor expansion in $v/N$.  The linear terms cancel and the quadratic remainder is bounded by
\[
        O\left(N^{-2}\sum_{k\ge0}e^{-c k/N}\right)=O(N^{-1}),
\]
uniformly for $v\in V$.  Hence
\begin{equation}\label{eq:UN-estimate}
        U_N(\tau,v)=1+O(N^{-1}),
\end{equation}
with the same bound after one $\tau$- or $v$-derivative, after Cauchy estimates in slightly enlarged compact sets.

It remains to evaluate the singular factor.  Write
\[
        \alpha=\frac{v}{b\tau},
        \qquad Q=e^{-h}.
\]
Then $e^{-v/N}=Q^{\alpha}$ and $e^{v/N}Q=Q^{1-\alpha}$, so
\[
        G_N(\tau,v)=
        \frac{\qpoch{Q^\alpha}{Q}\qpoch{Q^{1-\alpha}}{Q}}{\qpoch{Q}{Q}^2}.
\]
Using the $q$-gamma function
\[
        \Gamma_Q(z)=(1-Q)^{1-z}\frac{\qpoch{Q}{Q}}{\qpoch{Q^z}{Q}},
\]
we get
\[
        G_N(\tau,v)=\frac{1-Q}{\Gamma_Q(\alpha)\Gamma_Q(1-\alpha)}.
\]
As $Q\to1$ inside a fixed Stolz angle, $\Gamma_Q(z)\to\Gamma(z)$ locally uniformly in $z$ after taking reciprocals, and therefore, by Euler's reflection formula,
\[
        G_N(\tau,v)=
        (1-Q)\frac{\sin(\pi\alpha)}{\pi}
        +O(N^{-2})
        =\frac1N\frac{b\tau}{\pi}\sin\frac{\pi v}{b\tau}
          +O(N^{-2}),
\]
locally uniformly in $v$ and $\tau$.  Together with \eqref{eq:UN-estimate}, this proves \eqref{eq:H-sine-final}; differentiating with respect to $u$ is the same as applying $N\partial_v$, which gives \eqref{eq:Hu-sine-final}.  The stated derivative estimates follow from the same locally uniform holomorphic convergence and Cauchy estimates.

We next estimate $\psi(q_N(\tau))=\qpoch{q_N(\tau)}{q_N(\tau)}^3$.  We use the standard Euler--Maclaurin expansion, uniform in the same logarithmic window,
\begin{equation}\label{eq:qpoch-dilog-asym}
        \log\qpoch{a e^{-\gamma/N}}{e^{-b\tau/N}}
        =-\frac{N}{b\tau}\operatorname{Li}_2(a)+O(\log N),
\end{equation}
for $a$ in a compact subset of $\C\setminus[1,\infty)$, and the classical singular case
\begin{equation}\label{eq:singular-qpoch-asym}
        \log\qpoch{e^{-b\tau/N}}{e^{-b\tau/N}}
        =-\frac{\pi^2N}{6b\tau}+O(\log N).
\end{equation}
Both estimates follow by applying Euler--Maclaurin to the logarithm of the product; differentiating the same expansion gives the corresponding $\tau$-derivative with an $O(\log N)$ remainder after the leading term is removed.

Now factor
\[
        \qpoch{q}{q}=\qpoch{q^b}{q^b}
        \prod_{a=1}^{b-1}\qpoch{q^a}{q^b}.
\]
Since
\[
        \sum_{a=0}^{b-1}\operatorname{Li}_2(\omega^a)=\frac{\pi^2}{6b},
        \qquad
        \sum_{a=1}^{b-1}\operatorname{Li}_2(\omega^a)=\frac{\pi^2}{6b}-\frac{\pi^2}{6},
\]
\eqref{eq:qpoch-dilog-asym} and \eqref{eq:singular-qpoch-asym} give
\[
        \log\qpoch{q_N(\tau)}{q_N(\tau)}
        =-\frac{\pi^2N}{6b^2\tau}+O(\log N).
\]
Therefore
\begin{equation}\label{eq:psi-cusp-final}
        \log\psi(q_N(\tau))
        =-\frac{\pi^2N}{2b^2\tau}+O(\log N),
\end{equation}
again with the analogous derivative estimate after subtracting the displayed leading term.
Finally,
\[
        q_N(\tau)^{N(N+1)/2}
        =\omega^{N(N+1)/2}\exp\left(-\frac{\tau(N+1)}2\right).
\]
After fixing $N$ modulo $2b$, the phase $(-1)^N\omega^{N(N+1)/2}$ is fixed.  Combining this identity with \eqref{eq:psi-cusp-final} gives \eqref{eq:Lambda-final-asym} with
\[
        f_b(\tau)=\frac{\pi^2}{2b^2\tau}-\frac{\tau}{2},
\]
and all remaining factors absorbed into a holomorphic non-vanishing $A_N$ satisfying \eqref{eq:AN-bound}.
\end{proof}
\begin{remark}
The last display is the usual Dedekind eta asymptotic at a rational cusp.  The proof above is intentionally written only at the precision needed here: an exponential term, a non-zero holomorphic prefactor of at most polynomial size, and uniformity in a logarithmic boundary window.  The only external analytic input hidden in this shorthand is the standard uniform $q$-gamma/$q$-Pochhammer asymptotics as $Q\to1$ in a Stolz angle.  Equivalently, one may replace this paragraph by a citation to any standard uniform version of the McIntosh asymptotic expansion for $q$-shifted factorials~\cite{McIntoshQShifted}.
\end{remark}

\subsection{Matching the two exponential scales}
Choose a residue class $r$ modulo $b$ for which $E_r(\tau_0)\ne0$, as in Lemma~\ref{lem:tail-multiplier-nonzero}.  We now restrict $N$ to an arithmetic progression modulo $2b$ with $N\equiv r\pmod b$.  Put
\[
        \lambda_N(\tau):=\Lambda_N(\tau)E_r(\tau).
\]
At $\tau_0=\pi/b$,
\[
        f_b(\tau_0)=0,
        \qquad
        f_b'(\tau_0)=-1.
\]
Let
\[
        v_0:=\frac{3\pi}{4},
        \qquad
        \mu_0:=-e^{v_0}\sin v_0=-\frac{e^{3\pi/4}}{\sqrt2}.
\]
Since $S_b(\tau_0,v)=\sin v$, the pair $v_0,\mu_0$ is chosen so that
\[
        \sin v_0+\mu_0e^{-v_0}=0,
        \qquad
        \cos v_0-\mu_0e^{-v_0}=0.
\]

\begin{lemma}[One-variable matching]\label{lem:one-variable-matching}
For all sufficiently large $N$ in the chosen arithmetic progression there exists
\[
        \tau_N^*=\tau_0+O\left(\frac{\log N}{N}\right)
\]
such that
\begin{equation}\label{eq:lambda-match}
        \lambda_N(\tau_N^*)=\mu_0.
\end{equation}
Moreover, uniformly for bounded $\xi$,
\begin{equation}\label{eq:lambda-local}
        \lambda_N(\tau_N^*+\xi/N)=\mu_0e^{-\xi}\,(1+o(1)).
\end{equation}
\end{lemma}

\begin{proof}
By Lemmas~\ref{lem:tail-multiplier-nonzero} and \ref{lem:cusp-asymptotics-final}, $\lambda_N(\tau)=B_N(\tau)e^{Nf_b(\tau)}$, where $B_N$ is holomorphic, non-zero, and satisfies $\log|B_N|=O(\log N)$ and $\partial_\tau\log B_N=O(\log N)$ in the logarithmic window.  Choose a branch of $\log B_N$ after fixing the arithmetic progression.  The equation $\lambda_N(\tau)=\mu_0$ is equivalent to
\[
        Nf_b(\tau)+\log B_N(\tau)-\log\mu_0=0
\]
for a suitable fixed branch of $\log\mu_0$.  Since $f_b(\tau_0)=0$ and $f_b'(\tau_0)=-1$, Newton's theorem, or Rouch\'e's theorem on the circle $|\tau-\tau_0|=C\log N/N$ with $C$ sufficiently large, gives a solution $\tau_N^*$ in this circle.  Finally,
\[
        N\big(f_b(\tau_N^*+\xi/N)-f_b(\tau_N^*)\big)=-\xi+o(1),
\]
while $\log B_N(\tau_N^*+\xi/N)-\log B_N(\tau_N^*)=o(1)$.  This proves \eqref{eq:lambda-local}.
\end{proof}

\subsection{The two-variable lifting}
Define the normalized equations
\begin{equation}\label{eq:normalized-FN}
        F_{1,N}(\tau,v):=\frac{N}{\psi(q_N(\tau))}H_N(q_N(\tau),v/N),
        \qquad
        F_{2,N}(\tau,v):=\frac{1}{\psi(q_N(\tau))}\partial_uH_N(q_N(\tau),v/N).
\end{equation}
By Lemmas~\ref{lem:tail-asymptotics-final} and \ref{lem:cusp-asymptotics-final}, if $N\equiv r\pmod b$, then
\begin{align}
        F_{1,N}(\tau,v)
        &=S_b(\tau,v)+\lambda_N(\tau)e^{-v}+\mathcal E_{1,N}(\tau,v),\label{eq:F1-asym}\\
        F_{2,N}(\tau,v)
        &=\partial_vS_b(\tau,v)-\lambda_N(\tau)e^{-v}+\mathcal E_{2,N}(\tau,v),\label{eq:F2-asym}
\end{align}
where, uniformly in the logarithmic window and for $v$ in compact sets,
\begin{equation}\label{eq:F-error-bound}
        \mathcal E_{j,N}(\tau,v)
        =O\!\left(\frac{\log N}{N}\right)
         +O\!\left(\frac{|\Lambda_N(\tau)|}{N}\right),
        \qquad j=1,2.
\end{equation}
In particular, on every bounded $\xi$-window around a solution of the matching equation, i.e. after substituting
$\tau=\tau_N^*+\xi/N$ with $\lambda_N(\tau_N^*)=\mu_0$, Lemma~\ref{lem:one-variable-matching} gives
$|\lambda_N(\tau)|=O(1)$ and hence $\mathcal E_{j,N}=o(1)$ uniformly for bounded $\xi$ and for $v$ in compact sets.

\begin{theorem}[Root-of-unity lifting]\label{thm:root-unity-lifting-final}
Let $\omega$ be any root of unity.  Then there exist $q_N\in\Spec(\Theta)$, along an infinite arithmetic progression of $N$, such that
\[
        q_N\to\omega,
        \qquad |q_N|<1.
\]
More precisely, for all sufficiently large $N$ in a suitable arithmetic progression, there are
\[
        \tau_N=\frac\pi b+O\left(\frac{\log N}{N}\right),
        \qquad
        v_N=\frac{3\pi}{4}+o(1),
\]
such that, with
\[
        q_N=\omega e^{-\tau_N/N},
        \qquad
        u_N=v_N/N,
\]
one has
\[
        H_N(q_N,u_N)=\partial_uH_N(q_N,u_N)=0.
\]
Consequently $x_N=-q_N^{-N}e^{u_N}$ is a multiple zero of $\Theta(q_N,\cdot)$.
\end{theorem}

\begin{proof}
Let $r$ be chosen as in Lemma~\ref{lem:tail-multiplier-nonzero}, and let $\tau_N^*$ be given by Lemma~\ref{lem:one-variable-matching}.  Put
\[
        \tau=\tau_N^*+\xi/N.
\]
Then \eqref{eq:F1-asym}, \eqref{eq:F2-asym}, \eqref{eq:F-error-bound}, and \eqref{eq:lambda-local} imply that, uniformly for bounded $\xi$ and $v$ near $v_0=3\pi/4$,
\[
        (F_{1,N},F_{2,N})(\tau_N^*+\xi/N,v)
        \longrightarrow
        (F_1^\infty,F_2^\infty)(\xi,v),
\]
where
\[
        F_1^\infty(\xi,v)=\sin v+\mu_0e^{-\xi-v},
        \qquad
        F_2^\infty(\xi,v)=\cos v-\mu_0e^{-\xi-v}.
\]
By the definition of $\mu_0$, the limiting system has a zero at $(\xi,v)=(0,v_0)$.  Its Jacobian matrix there is
\[
        \begin{pmatrix}
        -\mu_0e^{-v_0} & \cos v_0-\mu_0e^{-v_0}\\
        \mu_0e^{-v_0} & -\sin v_0+\mu_0e^{-v_0}
        \end{pmatrix}
        =
        \begin{pmatrix}
        \sin v_0 & 0\\
        -\sin v_0 & -2\sin v_0
        \end{pmatrix},
\]
whose determinant is $-2\sin^2v_0\ne0$.  Therefore the zero is simple.  By the multidimensional Hurwitz theorem, equivalently by the two-variable Rouch\'e theorem on a sufficiently small polydisc around $(0,v_0)$, the exact system
\[
        F_{1,N}(\tau_N^*+\xi/N,v)=0,
        \qquad
        F_{2,N}(\tau_N^*+\xi/N,v)=0
\]
has a solution $(\xi_N,v_N)$ tending to $(0,v_0)$.  Set
\[
        \tau_N:=\tau_N^*+\xi_N/N,
        \qquad
        q_N:=\omega e^{-\tau_N/N},
        \qquad
        u_N:=v_N/N.
\]
Then $H_N(q_N,u_N)=\partial_uH_N(q_N,u_N)=0$ by \eqref{eq:normalized-FN}.  Since $\Re\tau_N=\pi/b+o(1)>0$, one has $|q_N|<1$; and clearly $q_N\to\omega$.  Finally Proposition~\ref{prop:moving-window-double} converts this pair of equations into a genuine double zero $x_N=-q_N^{-N}e^{u_N}$ of $\Theta(q_N,\cdot)$.
\end{proof}

\begin{corollary}[Full unit-circle accumulation]\label{cor:full-unit-circle-accumulation}
Every point of the unit circle is an accumulation point of the spectrum of the partial theta function:
\[
        \overline{\Spec(\Theta)}\cap\T=\T.
\]
\end{corollary}

\begin{proof}
Theorem~\ref{thm:root-unity-lifting-final} gives spectral points tending to every root of unity.  Since roots of unity are dense on $\T$, for any $\zeta\in\T$ choose roots of unity $\omega_j\to\zeta$ and then choose $q_j\in\Spec(\Theta)$ with $|q_j-\omega_j|<1/j$.  Then $q_j\to\zeta$.
\end{proof}

\begin{remark}[Density]
The proof is local and produces spectral points in logarithmic boundary windows near roots of unity.  It does not yet count all such points in macroscopic arcs.  The truncation-spectral factor $\Psi_m$ has limiting angular density $d\theta/(2\pi)$ by Theorem~\ref{thm:equidistribution}; the same density is the natural conjecture for the genuine spectral points produced above, but proving a counting theorem would require tracking all branches of the matching equation $\lambda_N(\tau)=\mu$ uniformly in $\arg\omega$ and in the residue class.
\end{remark}

\section{Interior approximation: finite spectra, caustics, and bounded roots}
Let
\[
  \Spec(\Th)=\{q\in\D:\exists z\in\C,\ \Th(q,z)=\partial_z\Th(q,z)=0\}
\]
be the true complex spectrum.  For the ordinary truncation put
\[
  \Th_n(q,z)=\sum_{j=0}^{n}q^{j(j+1)/2}z^j,
  \qquad D_n^{\tr}(q)=\disc_z\Th_n(q,z).
\]
The roots of $D_n^{\tr}$ are branch points of the finite algebraic curve $\Th_n(q,z)=0$ under the projection to the $q$-plane.  They should not be confused with $\Spec(\Th)$.

The reason is simple but important.  The convergence $\Th_n\to\Th$ is uniform only on compact subsets of $\D\times\C$.  Thus it controls double roots whose $z$-coordinates remain bounded, but it says nothing directly about collisions escaping to $|z|=\infty$.  Such escaping collisions can produce visible arcs or caustics in finite discriminant pictures.

\begin{proposition}[bounded double roots converge to true spectral values]\label{prop:bounded}
Let $n_k\to\infty$, $q_k\to q_*$ with $|q_*|<1$, and suppose that $z_k$ is a double zero of $\Th_{n_k}(q_k,\cdot)$.  If $z_k$ remains bounded, then, after passing to a subsequence, $z_k\to z_*$ and
\[
  \Th(q_*,z_*)=0,\qquad \partial_z\Th(q_*,z_*)=0.
\]
In particular $q_*\in\Spec(\Th)$.
\end{proposition}

\begin{proof}
Choose $0<\delta<1$ and $R<\infty$ such that eventually $|q_k|\leq\delta$ and $|z_k|\leq R$.  The series defining $\Th$ and $\partial_z\Th$ converge uniformly on the polydisc $|q|\leq\delta$, $|z|\leq R$.  Hence $\Th_{n_k}\to\Th$ and $\partial_z\Th_{n_k}\to\partial_z\Th$ uniformly on this polydisc.  Passing to the limit in
\[
  \Th_{n_k}(q_k,z_k)=0,
  \qquad
  \partial_z\Th_{n_k}(q_k,z_k)=0
\]
gives the claim.
\end{proof}

\begin{proposition}[local persistence of non-degenerate spectral points]\label{prop:persistence}
Let $(q_*,z_*)$ satisfy
\[
  \Th(q_*,z_*)=0,
  \qquad
  \partial_z\Th(q_*,z_*)=0,
  \qquad |q_*|<1.
\]
Assume that the Jacobian of the map
\[
  (q,z)\longmapsto (\Th(q,z),\partial_z\Th(q,z))
\]
is non-zero at $(q_*,z_*)$.  Then, for all sufficiently large $n$, there exists a unique pair $(q_n,z_n)$ near $(q_*,z_*)$ such that
\[
  \Th_n(q_n,z_n)=0,
  \qquad
  \partial_z\Th_n(q_n,z_n)=0,
\]
and $(q_n,z_n)\to(q_*,z_*)$.
\end{proposition}

\begin{proof}
This is the analytic implicit function theorem applied to the perturbation
\[
  (\Th_n,\partial_z\Th_n)=(\Th,\partial_z\Th)+o(1)
\]
uniformly on a small polydisc around $(q_*,z_*)$.  The non-vanishing Jacobian gives local invertibility, and the uniform convergence gives a unique zero of the perturbed system for all large $n$.
\end{proof}

\begin{remark}
Propositions~\ref{prop:bounded} and~\ref{prop:persistence} give the practical rule used below.  In a fixed disk $|q|\leq r<1$, truncation branch points are useful only after one records the corresponding double root $z$.  Bounded $z$-branches are genuine candidates for the true spectrum; escaping $z$-branches are caustics of the finite approximants.
\end{remark}

\section{The exact even-degree factorization}
For even degrees the ordinary truncation discriminant has a useful structural split.  Let $n=2m$ and put $q=t^2$, $z=t^{-(2m+1)}y$.  Then
\[
  R_m(t,y):=t^{m^2}\Th_{2m}(t^2,t^{-(2m+1)}y)
  =\sum_{j=0}^{2m}t^{(j-m)^2}y^j .
\]
This polynomial is palindromic:
\[
  R_m(t,y)=y^{2m}R_m(t,y^{-1}).
\]
Consequently
\[
  R_m(t,y)=y^m S_m(t,y+y^{-1})
\]
for a polynomial $S_m$ of degree $m$ in the second variable.  The discriminant of a palindromic polynomial gives, up to a non-zero monomial and a numerical constant,
\[
  \disc_y R_m(t,y)=R_m(t,1)R_m(t,-1)\,\disc_u(S_m(t,u))^2 .
\]
Thus the reduced truncation discriminant has the form
\[
  \widehat D^{\tr}_{2m}(q)=C_m(q)B_m(q)^2.
\]
The factor
\[
  C_m(t^2)=R_m(t,1)R_m(t,-1)
\]
corresponds to the central reciprocal positions $y=\pm1$, i.e.
\[
  z=\pm q^{-(2m+1)/2}.
\]
Therefore $C_m$ is an escaping central factor for every fixed $|q|<1$.  It is not a direct approximation to the interior spectrum.  The factor $B_m$ contains the stable approximants to genuine spectral values, but it also contains escaping caustic points.  Hence the algebraic factorization is helpful, but it does not replace the bounded-root test.

\section{A truncation-seeded search in $|q|\leq 0.8$}
The following experiment was designed to avoid over-interpreting finite discriminant pictures.

\begin{enumerate}
\item Compute the reduced truncation discriminants for $n=8,10,12,14$.
\item Keep roots with $|q|<0.82$ and compute a corresponding double root $z$ by solving
\[
  \Th_n(q,z)=0,
  \qquad
  \partial_z\Th_n(q,z)=0.
\]
\item Use these pairs $(q,z)$ as starting points for Newton's method applied to the infinite system
\[
  F_1(q,z)=\Th(q,z)=0,
  \qquad
  F_2(q,z)=\partial_z\Th(q,z)=0.
\]
\item Cluster the converged solutions and retain those with $|q|\leq0.8$.
\end{enumerate}

The Newton method used the analytic Jacobian
\[
\begin{pmatrix}
\partial_q\Th & \partial_z\Th\\
\partial_q\partial_z\Th & \partial_z^2\Th
\end{pmatrix},
\]
where
\[
  \partial_q\Th(q,z)=\sum_{j\geq1}\frac{j(j+1)}2 q^{j(j+1)/2-1}z^j,
\]
\[
  \partial_q\partial_z\Th(q,z)
  =\sum_{j\geq1}j\frac{j(j+1)}2q^{j(j+1)/2-1}z^{j-1},
\]
and
\[
  \partial_z^2\Th(q,z)
  =\sum_{j\geq2}j(j-1)q^{j(j+1)/2}z^{j-2}.
\]
All listed solutions were refined until the maximum of $|\Th(q,z)|$ and $|\partial_z\Th(q,z)|$ was below $10^{-30}$; in practice the residuals were typically far smaller.  Since the largest listed $|z|$ is about $16.2$, the tail of the partial theta series is extremely small for $|q|\leq0.8$ after a moderate number of terms.  This gives a reliable numerical calculation, although not yet a fully rigorous exclusion proof for unlisted points.

\begin{figure}[htbp]
\centering
\includegraphics[width=.82\linewidth]{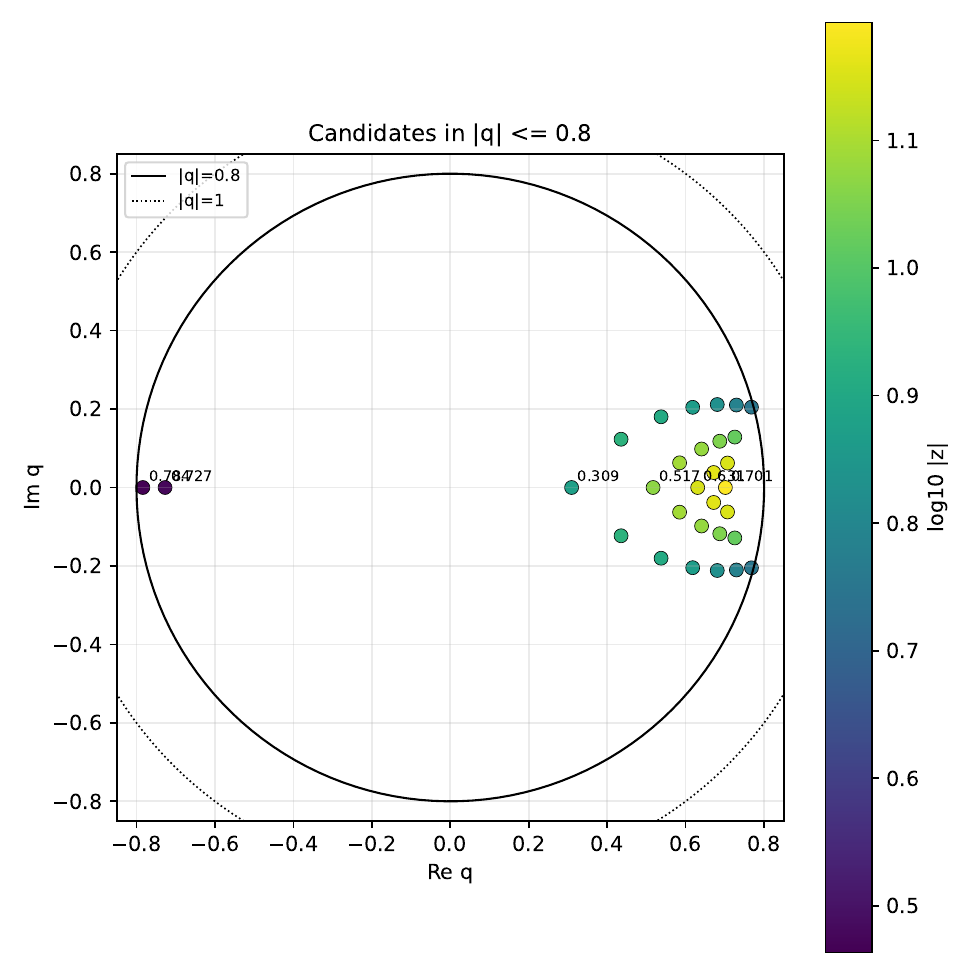}
\caption{Newton-refined spectral candidates of the infinite partial theta function in $|q|\leq0.8$, colored by $\log_{10}|z|$.  The output is discrete.  It does not resemble a circular arc filled densely by spectral points.}
\label{fig:candidates-r08}
\end{figure}

\begin{figure}[htbp]
\centering
\includegraphics[width=.82\linewidth]{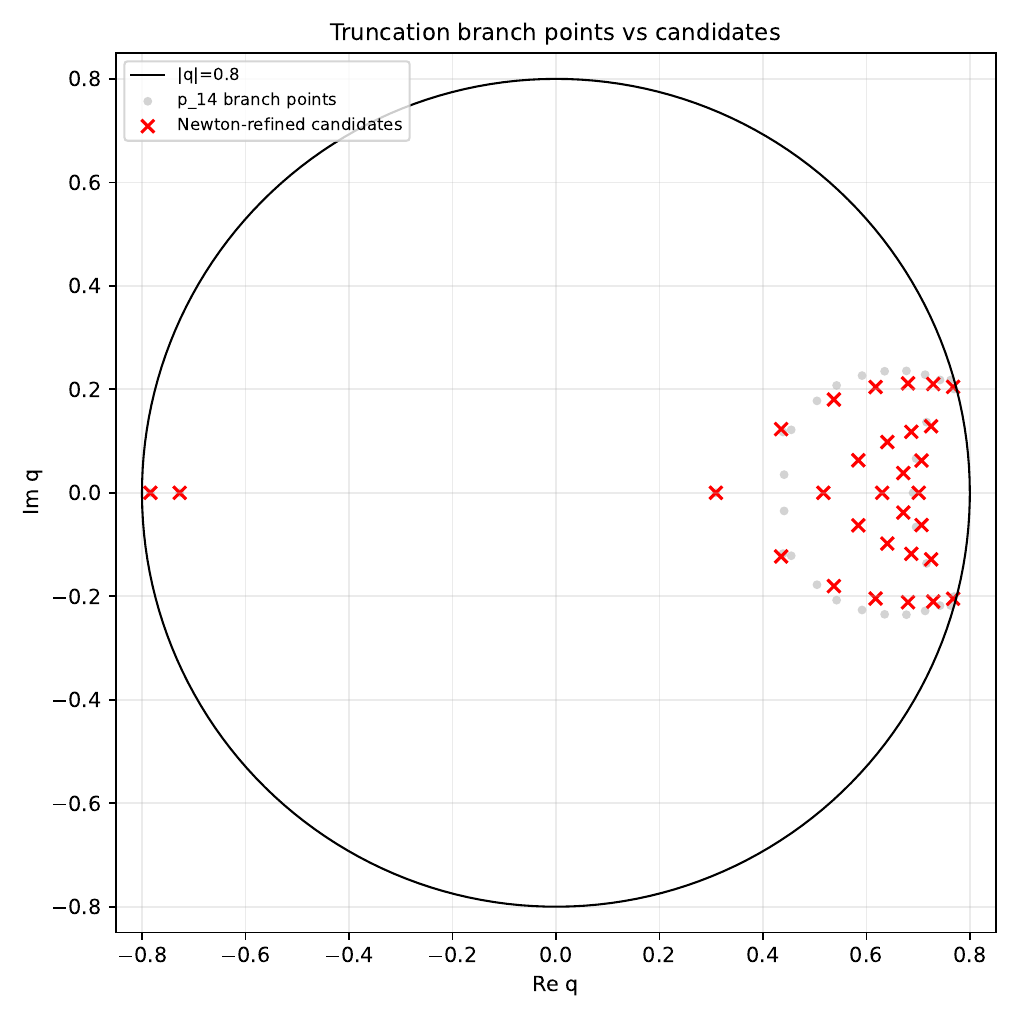}
\caption{Degree $14$ truncation branch points in $|q|\leq0.8$ shown in gray, with Newton-refined infinite spectral candidates shown in red.  Many finite branch points do not refine to distinct nearby points of the true spectrum; they should be interpreted as finite-truncation branch geometry or caustics unless the associated $z$-branches remain controlled.}
\label{fig:trunc-vs-candidates}
\end{figure}

\section{Numerical candidates in $|q|\leq0.8$}
The truncation-seeded Newton refinement produced the following candidate spectral values.  Complex conjugation symmetry is visible, as expected from the real coefficients of the defining equations.

{\scriptsize
\begin{longtable}{rcccc}
\caption{Newton-refined spectral candidates in $|q|\leq0.8$.}\label{tab:r08-candidates}\\
\toprule
No. & $q$ & $|q|$ & corresponding double root $z$ & $|z|$\\
\midrule
\endfirsthead
\toprule
No. & $q$ & $|q|$ & corresponding double root $z$ & $|z|$\\
\midrule
\endhead
1 & $0.309249338600$ & $0.309249338600$ & $-7.503255964244$ & $7.503256$ \\
2 & $0.435318495824-0.123044008552i$ & $0.452373762333$ & $-5.963923719620-6.104775174236i$ & $8.534440$ \\
3 & $0.435318495824+0.123044008552i$ & $0.452373762333$ & $-5.963923719620+6.104775174236i$ & $8.534440$ \\
4 & $0.516959359788$ & $0.516959359788$ & $-11.713168218924$ & $11.713168$ \\
5 & $0.537338919471-0.180327336927i$ & $0.566790140018$ & $-2.468304127681-7.661246127761i$ & $8.049051$ \\
6 & $0.537338919471+0.180327336927i$ & $0.566790140018$ & $-2.468304127681+7.661246127761i$ & $8.049051$ \\
7 & $0.584502393285-0.062921301797i$ & $0.587879356650$ & $-10.648704780600-6.172536070113i$ & $12.308335$ \\
8 & $0.584502393285+0.062921301797i$ & $0.587879356650$ & $-10.648704780600+6.172536070113i$ & $12.308335$ \\
9 & $0.630628316063$ & $0.630628316063$ & $-14.068512932540$ & $14.068513$ \\
10 & $0.640585293973-0.098117777816i$ & $0.648056029351$ & $-7.446041156099-9.311633906595i$ & $11.922670$ \\
11 & $0.640585293973+0.098117777816i$ & $0.648056029351$ & $-7.446041156099+9.311633906595i$ & $11.922670$ \\
12 & $0.617908177237-0.204391067437i$ & $0.650835020528$ & $-0.077489573087-7.339690120098i$ & $7.340099$ \\
13 & $0.617908177237+0.204391067437i$ & $0.650835020528$ & $-0.077489573087+7.339690120098i$ & $7.340099$ \\
14 & $0.671422275226-0.038183681463i$ & $0.672507148810$ & $-13.360596934577-5.523733206370i$ & $14.457426$ \\
15 & $0.671422275226+0.038183681463i$ & $0.672507148810$ & $-13.360596934577+5.523733206370i$ & $14.457426$ \\
16 & $0.686934155443-0.117904588691i$ & $0.696979214861$ & $-4.359842474117-10.346690340867i$ & $11.227744$ \\
17 & $0.686934155443+0.117904588691i$ & $0.696979214861$ & $-4.359842474117+10.346690340867i$ & $11.227744$ \\
18 & $0.701265070083$ & $0.701265070083$ & $-15.578168997259$ & $15.578169$ \\
19 & $0.706704199173-0.062367295983i$ & $0.709450847302$ & $-10.891765132269-9.088484903579i$ & $14.185595$ \\
20 & $0.706704199173+0.062367295983i$ & $0.709450847302$ & $-10.891765132269+9.088484903579i$ & $14.185595$ \\
21 & $0.680536370502-0.211558604123i$ & $0.712661767288$ & $1.337568051901-6.550754439619i$ & $6.685916$ \\
22 & $0.680536370502+0.211558604123i$ & $0.712661767288$ & $1.337568051901+6.550754439619i$ & $6.685916$ \\
23 & $-0.727133325456$ & $0.727133325456$ & $-2.991115175906$ & $2.991115$ \\
24 & $0.725242305555-0.128589594053i$ & $0.736553925701$ & $-1.935836031954-10.299221477311i$ & $10.479572$ \\
25 & $0.725242305555+0.128589594053i$ & $0.736553925701$ & $-1.935836031954+10.299221477311i$ & $10.479572$ \\
26 & $0.729328603046-0.210226851486i$ & $0.759022753484$ & $2.153977346974-5.741122436580i$ & $6.131892$ \\
27 & $0.729328603046+0.210226851486i$ & $0.759022753484$ & $2.153977346974+5.741122436580i$ & $6.131892$ \\
28 & $-0.783742093195$ & $0.783742093195$ & $2.906784175410$ & $2.906784$ \\
29 & $0.767739450253-0.204718120811i$ & $0.794564895061$ & $2.622532669757-5.028188499010i$ & $5.671010$ \\
30 & $0.767739450253+0.204718120811i$ & $0.794564895061$ & $2.622532669757+5.028188499010i$ & $5.671010$ \\
\bottomrule
\end{longtable}
}

The positive real values found in this disk are approximately
\[
  0.309249338600,
  \quad
  0.516959359788,
  \quad
  0.630628316063,
  \quad
  0.701265070083,
\]
and the negative real values found are approximately
\[
  -0.727133325456,
  \quad
  -0.783742093195.
\]
The remaining candidates occur in conjugate pairs.  This list should be interpreted as a high-quality numerical approximation to the portion of $\Spec(\Th)$ detected by truncations up to degree $14$ and then refined against the infinite equations.

\section{Radial monodromy from the small circle}
\label{sec:radial-monodromy}
The numerical monodromy information in a compact subdisk is path-dependent unless the base point and the connecting path are fixed.  The convention used in this section is more intrinsic than the earlier vertical-horizontal convention.  Let
\[
        q_*=|q_*|e^{i\theta}
\]
be a spectral point with \(|q_*|>0.1\).  We start at
\[
        p(q_*):=0.1e^{i\theta}
\]
on the small circle \(|q|=0.1\), label the zeros of \(\Theta(p(q_*),z)\) by increasing modulus, and continue the roots along the straight radial segment from \(p(q_*)\) to \(q_*\).  Equivalently, for small \(|q|\) on that ray the label is characterized by
\[
        z_k(re^{i\theta})=-r^{-k}e^{-ik\theta}(1+O(r)),
        \qquad r=0.1.
\]
The disk \(|q|<0.1\) is numerically free of spectral values, and in this disk the roots are well separated in modulus, so this gives a stable initial labelling.  For a negative real spectral point we use the intrinsic negative base point
\[
        p(q_*)=-0.1,
\]
not a path from \(0.1\) through the origin.  This is essential, since \(q=0\) is a completely degenerate limiting point for the root labelling.

The computation was carried out as follows.  For a listed spectral point \((q_*,z_*)\), a point \(q\) very close to \(q_*\) on the incoming radial segment was chosen, the two local roots near \(z_*\) were separated using the square-root approximation coming from
\[
        \Theta(q,z)=0,
        \qquad
        \Theta(q_*,z_*)=\Theta_z(q_*,z_*)=0,
\]
and the two branches were continued backwards to \(p(q_*)\).  The resulting endpoints were matched with the small-circle roots.  As a check, the same labels were obtained by forward continuation of the relevant small-circle roots for the points displayed below.  The local singularities are simple in all tested cases, i.e.
\[
        \Theta_q(q_*,z_*)\neq0,
        \qquad
        \Theta_{zz}(q_*,z_*)\neq0,
\]
so the local monodromy is a transposition.  The table records which two small-circle labels collide.

\begin{center}
\footnotesize
\begin{tabular}{@{}c l l c@{}}
\toprule
point & \(q_*\) & double root \(z_*\) & radial collision labels\\
\midrule
positive real & $0.309249$ & $-7.503256$ & $(1\;2)$\\
positive real & $0.516959$ & $-11.713168$ & $(3\;4)$\\
positive real & $0.630628$ & $-14.068513$ & $(5\;6)$\\
positive real & $0.701265$ & $-15.578169$ & $(7\;8)$\\
\midrule
upper half-plane & $0.435318+0.123044i$ & $-5.963924+6.104775i$ & $(2\;3)$\\
upper half-plane & $0.537339+0.180327i$ & $-2.468304+7.661246i$ & $(3\;4)$\\
upper half-plane & $0.584502+0.062921i$ & $-10.648705+6.172536i$ & $(4\;5)$\\
upper half-plane & $0.640585+0.098118i$ & $-7.446041+9.311634i$ & $(5\;6)$\\
upper half-plane & $0.617908+0.204391i$ & $-0.077490+7.339690i$ & $(3\;5)$\\
upper half-plane & $0.680536+0.211559i$ & $1.337568+6.550754i$ & $(3\;6)$\\
upper half-plane & $0.729329+0.210227i$ & $2.153977+5.741122i$ & $(3\;7)$\\
upper half-plane & $0.767739+0.204718i$ & $2.622533+5.028188i$ & $(2\;8)$\\
\midrule
negative real & $-0.727133$ & $-2.991115$ & $(2\;4)$\\
negative real & $-0.783742$ & $ 2.906784$ & $(3\;5)$\\
\bottomrule
\end{tabular}
\end{center}

The first two non-real dominating points have the same labels as in the vertical-horizontal experiment, namely \((2\;3)\) and \((3\;4)\).  Starting with the next dominating points the radial convention gives different labels: for instance
\[
        0.617908177237+0.204391067437i
        \quad\leadsto\quad (3\;5),
\]
and
\[
        0.680536370502+0.211558604123i
        \quad\leadsto\quad (3\;6).
\]
This is not a contradiction; it shows that the global identification of sheets depends on the path used to reach the local branch point.  The radial convention is preferable for a global organization of the disk because it always begins from the same small circle \(|q|=0.1\), where the ordering by modulus is canonical.

\begin{figure}[htbp]
\centering
\includegraphics[width=.90\linewidth]{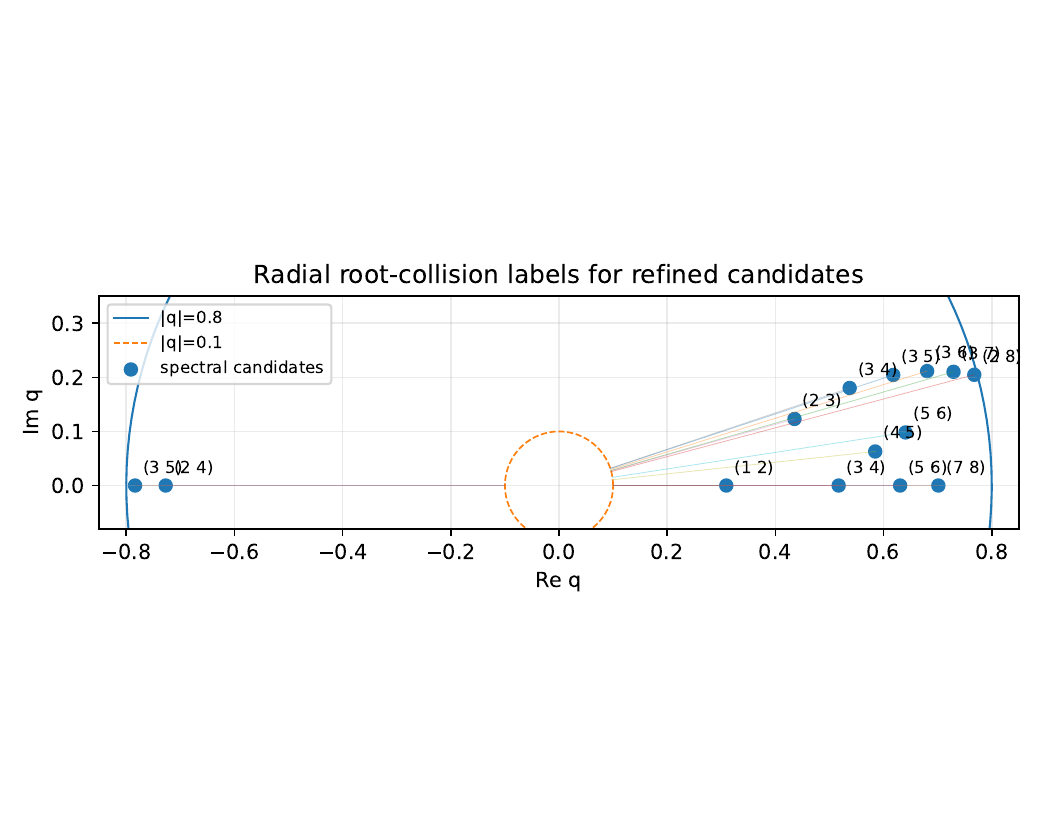}
\caption{Radial root-collision labels for the refined candidates in \(|q|\leq0.8\).  For each displayed point \(q_*\), the base point is \(0.1q_*/|q_*|\), except on the negative real axis where the base point is \(-0.1\).  The label attached to a point indicates which two roots at the corresponding small-circle base point collide when continued radially to \(q_*\).}
\label{fig:radial-monodromy}
\end{figure}

The positive real entries agree with the known real-root theory: along \((0,1)\) the positive real spectral values correspond to consecutive pairs
\[
        (1\;2),\quad (3\;4),\quad (5\;6),\quad (7\;8),\ldots .
\]
These statements should be compared with the work of Kostov--Shapiro and Kostov on the real spectrum and double zeros \cite{KostovShapiro,KostovSpectrum,KostovAsymptotics,KostovDoubleZeros}.  On the negative real axis the small-negative-
\(q\) labelling alternates signs:
\[
        z_k(-r)=(-1)^{k+1}r^{-k}(1+O(r)).
\]
Thus \((2\;4)\) is the collision of the first two negative roots at the base point \(-0.1\), while \((3\;5)\) is the collision of the first two positive roots.  This is why the negative real labels differ by two rather than by one.

\section{Rational directions and radial monodromy}
\label{sec:rational-radial}
The same radial convention suggests a more general rule near rational directions.  Let
\[
        \theta=2\pi\frac{a}{b},\qquad (a,b)=1,
\]
and consider the small base point \(p=0.1e^{i\theta}\).  The small-
\(q\) asymptotics along this ray are
\[
        z_k(re^{i\theta})\sim -r^{-k}e^{-ik\theta}.
\]
Hence the arguments of the leading terms depend only on \(k\bmod b\).  The roots split into \(b\) angular classes, and inside each class the natural radial ordering is
\[
        r_0,\quad r_0+b,\quad r_0+2b,\quad\ldots .
\]
This gives the following radial analogue of the earlier rational-angle heuristic.

\begin{problem}[radial monodromy near a rational direction]
Fix a rational direction \(e^{2\pi ia/b}\).  For spectral points approaching this direction from inside the unit disk, label the roots at the small base point \(0.1e^{2\pi ia/b}\) by increasing modulus.  Is the spectrum near this direction organized into collections indexed by residue classes modulo \(b\), with the simplest local monodromies
\[
        \bigl(r+mb,\; r+(m+1)b\bigr),
        \qquad m=0,1,2,\ldots ?
\]
For \(b=1\) this gives the adjacent positive-real pairs \((1\;2),(3\;4),(5\;6),\ldots\).  For \(b=2\), i.e. the negative real direction, it gives the parity-preserving pairs \((2\;4),(3\;5),\ldots\) after choosing the base point \(-0.1\).  For \(b=4\), i.e. the two imaginary directions, it predicts pairs whose labels differ by four, such as \((1\;5),(2\;6),(3\;7),(4\;8)\), in the corresponding residue-class collections.
\end{problem}

\begin{figure}[htbp]
\centering
\includegraphics[width=.90\linewidth]{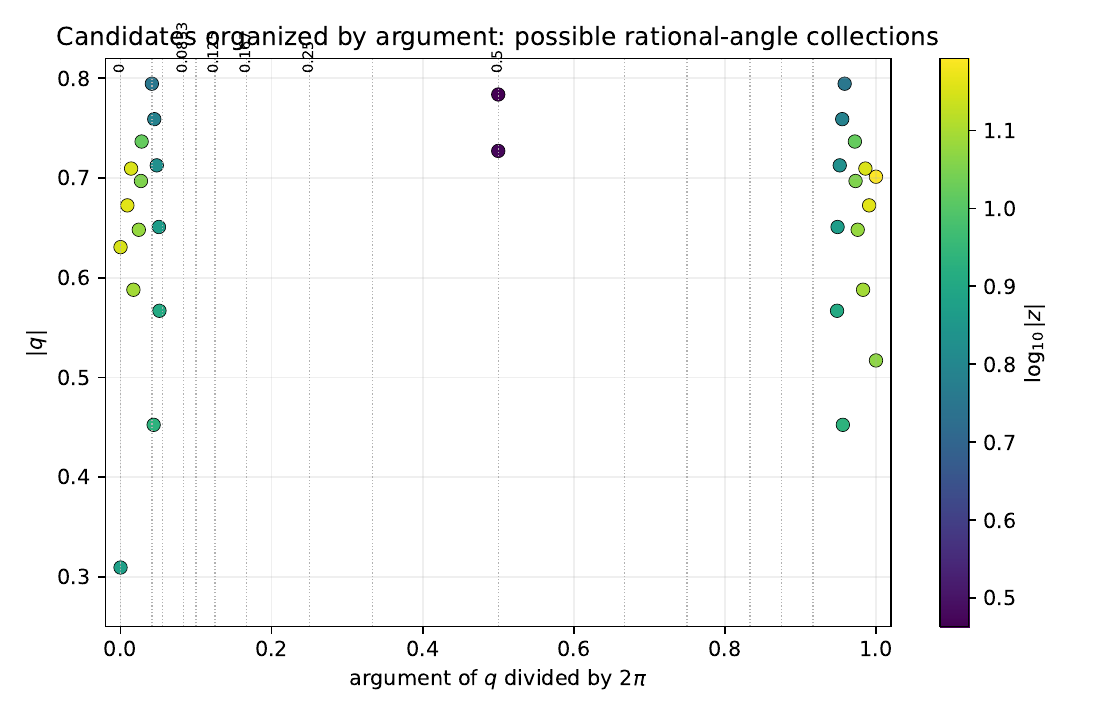}
\caption{The refined candidates plotted by argument and modulus.  Dotted vertical guide lines mark several rational arguments.  The present \(|q|\leq0.8\) list contains no candidates close to the imaginary axis; its non-real candidates remain comparatively close to the positive real direction.  Thus the predictions for directions such as \(\arg q=\pm\pi/2\) should be tested using candidates closer to the unit circle.}
\label{fig:argument-collections}
\end{figure}

The last point is important.  In the current \(|q|\leq0.8\) table, the largest arguments of upper-half-plane candidates are only about
\[
        0.32\text{ radians }\approx 18.5^\circ,
\]
so the data do not yet seriously test the imaginary direction.  The farthest-from-real points available in this table already show that radial monodromy need not be adjacent in the original positive-real ordering: the points
\[
\begin{aligned}
0.617908177237+0.204391067437i&\leadsto (3\;5),\\
0.680536370502+0.211558604123i&\leadsto (3\;6),\\
0.729328603046+0.210226851486i&\leadsto (3\;7),\\
0.767739450253+0.204718120811i&\leadsto (2\;8).
\end{aligned}
\]
A meaningful numerical test near \(\arg q=\pi/2\) will require spectral candidates much closer to \(|q|=1\), probably generated from the boundary-window equations of Sections~5--7 rather than from low-degree truncation seeds.

\begin{problem}[near-imaginary numerical test]
Use the moving-window system
\[
        H_N(q,u)=\partial_uH_N(q,u)=0
\]
with \(q\) in a boundary layer near \(i\) or \(-i\) to compute genuine spectral points of \(\Theta\) with arguments close to \(\pm\pi/2\).  The radial residue-class heuristic predicts that the first observable monodromies should exchange root labels differing by four.
\end{problem}

\section{Interpretation}
The computation changes the interpretation of the finite pictures in an important way.

First, the true spectrum inside $|q|\leq0.8$ appears as a discrete finite set of points, not as arcs.  This is consistent with the general analytic expectation that the spectrum in the open unit disk is discrete.  The finite truncation spectra, by contrast, contain many more branch points, some arranged in visually coherent curves.

Second, the ordinary truncations give a concrete diagnostic missing from plots of discriminant roots alone.  A point $q$ should not be judged by its position in the $q$-plane only; one must compute the corresponding double root $z$.  If $z$ remains bounded along an increasing-degree sequence, Proposition~\ref{prop:bounded} makes the point a genuine spectral candidate.  If $z$ escapes, the point belongs to a finite-degree caustic.

Third, the factorization
\[
  \widehat D_{2m}^{\tr}(q)=C_m(q)B_m(q)^2
\]
explains why a clean factor split does not automatically identify the true spectrum.  The factor $C_m$ is essentially central and escaping.  The factor $B_m$ contains genuine bounded approximants, but it also contains off-central escaping collisions.  The true/caustic separation is therefore not a factorization problem alone; it is a problem in the joint $(q,z)$-space.

\section{Toward a certified theorem in $|q|\leq0.8$}
The present computation suggests the following certification program.

\begin{problem}[certification in a fixed disk]
Prove that the table above gives all points of $\Spec(\Th)$ in $|q|\leq0.8$.
\end{problem}

A possible route is as follows.

\begin{enumerate}
\item Use interval Newton or Krawczyk operators in $\C^2$ to enclose each listed solution $(q,z)$ in a small polydisc containing a unique solution of
\[
  \Th(q,z)=\partial_z\Th(q,z)=0.
\]
\item Choose $R>0$ large enough and use a box subdivision of
\[
  |q|\leq0.8,
  \qquad
  |z|\leq R
\]
combined with interval estimates to exclude all other bounded solutions.
\item Prove an a priori exclusion for $|z|>R$ when $|q|\leq0.8$, or else show that any such solutions would have to arise as limits of an explicitly described escaping caustic family.
\end{enumerate}

The third step is the mathematically deepest one.  Without it, one obtains a reliable list of bounded-root spectral candidates, but not a rigorous proof of completeness in the disk.  Nevertheless, the truncation-seeded Newton refinement already gives a much more faithful picture of the true spectrum than raw discriminant plots of truncations or Jensen polynomials.

\section{Interior conclusions and certification problem}
For a fixed subdisk such as $|q|\leq0.8$, the safest approximation to the spectrum of the partial theta function is not the full spectrum of truncations or Jensen polynomials.  The correct numerical object is the subset of finite branch points whose associated double roots in the $z$-plane remain bounded and which refine to solutions of the infinite equations.

Using this procedure with ordinary truncations up to degree $14$ gives the explicit list above.  The result supports the skepticism about circle-like arcs: such arcs appear naturally in finite discriminant pictures, but after bounded-root filtering and Newton refinement they do not survive as arcs of the true spectrum inside $|q|\leq0.8$.

\section{Compatibility of the two pictures}
The boundary and interior results should be read together.  The unit-circle theorem says that every point of \(\partial\D\) is an accumulation point of spectral values.  This does not imply, and is not meant to imply, that the spectrum fills arcs inside \(\D\).  In compact subsets of the open disk, the spectrum is locally finite.  Finite truncations and Jensen polynomials nevertheless show arc-like branch loci, because their discriminants record both bounded collisions, which may converge to true spectral points, and escaping collisions, which form caustics in the \(q\)-plane.

Thus the correct principle is: in the boundary problem one uses scaling limits in which the double root escapes in a controlled way, while in the fixed-subdisk problem a finite branch point is useful only after one computes and controls the corresponding double root.  This reconciles the two sets of figures and gives a coherent project: prove unit-circle accumulation, and certify finite lists of interior spectral points by bounded-root and interval/Rouch\'e methods.

\section*{Acknowledgements}
The author thanks Vladimir P. Kostov and Jens Forsg\aa rd for many discussions and for their earlier work and notes on the partial theta function, truncations, Jensen polynomials, and related spectra. I am additionally grateful to Professor Kostov for sharing with me his text \cite{GatiKostov} prior to its publication.

\end{document}